\documentclass[10pt,fleqn]{article}
\usepackage{array,delarray,calc,amssymb,amsmath,amsthm,mathrsfs,graphicx}

\let\a=\alpha

\let\c=\chi

\let\e=\varepsilon

\let\f=\varphi
\let\F=\Phi
\let\g=\gamma
\let\G=\Gamma

\let\l=\lambda

\let\m=\mu

\let\N=\nabla

\let\r=\rho
\let\s=\sigma
\let\S=\Sigma

\let\vp=\varpi

\let\W=\Omega
\let\x=\xi
\let\X=\Xi

\let\y=\psi
\let\Y=\Psi

\def\cL{{\cal L}}

\def\cN{{\cal N}}

\def\gg{{\mathfrak{g}}}

\def\gso{{\mathfrak{so}}}

\def\s1{S^{n-1}}

\let\os=\oplus

\def\endo{\operatorname{End}}
\def\eye{\sqrt{-1}}

\def\id{\operatorname{id}}

\def\proj{\operatorname{Proj}}

\def\spin{\operatorname{Spin}}

\def\dc{\mathpalette\dC{}}
\def\dC#1{\ooalign{$\hfil\raisebox{0.2ex}{$#1\mkern1mu/$}\hfil$\cr$#1\nabla$}}

\let\ptl=\partial
\def\pt{\partial_t}

\def\spm{\begin{pmatrix}\frac{\X}{\eye}\f\\\f\end{pmatrix}}
\def\spfy{\begin{pmatrix}\y\\\f\end{pmatrix}}
\def\Lt{\mathcal{L}_{_T}}

\def\sideremark#1{\ifvmode\leavevmode\fi%
 \vadjust{%
  \vbox to0pt{%
   \vss\hbox to 0pt{%
    \hskip\hsize\hskip1em\vbox{%
     \hsize2cm\tiny\raggedright\pretolerance10000\noindent #1\hfill%
    }%
    \hss%
   }%
   \vbox to\baselineskip{\vfil}\vss%
  }%
 }%
}

\newtheorem{theorem}{Theorem}[section]
\newtheorem{lemma}[theorem]{Lemma}
\newtheorem{remark}[theorem]{Remark}

\title{Powers of the Dirac Operator on $S^1\times S^{n-1}$}
\author{Doojin Hong}
\date{\today}

\begin{document}
\maketitle
\begin{abstract}
We give explicit formulas for the spectra of the intertwinors on the spinor bundle over $S^1 \times S^{n-1}$, $n$ even, with the standard Lorentzian metric $g=-g_{{}_{S^1}}+g_{{}_{S^{n-1}}}$. As a special case, we construct conformally covariant differential operators of all odd orders with the leading term a power of the Dirac operator.
\end{abstract}
\section{Introduction}
In \cite{BOO:96}, Branson, \'Olafsson and \O rsted introduced spectrum generating technique. The method has been successfully applied to some multiplicity one and two cases (\cite{BH:06}, \cite{BOO:96}, \cite{Hong:11}). Branson and \O rsted used the method to construct conformally covariant differential operators of all odd orders on the spinor bundle over the standard sphere $S^n$ in \cite{BO:06}.  
In this paper, we apply the technique to the spinor bundle over $S^1\times S^{n-1}$ and present the spectral function for the intertwinors and construct conformally invariant differential operators of all odd orders with the leading term a power of the Dirac operator. 
\section{Spinors on $S^1\times S^{n-1}$}
Let $\S$ be the spinor bundle on the odd dimensional ($n$ even) sphere $\s1$ with the standard spin structure. The fundamental tensor-spinor $\g$ is a section of $T^*\s1\otimes\endo(\S)$ satisfying  the Clifford relation
\begin{equation}\label{Cliff}
\g_i\g_j+\g_j\g_i=-2g_{ij}\id_\S, 
\end{equation}
where $g$ is the metric tensor on $\s1$. The Dirac operator $D$ can be expressed as $D=\g^i\N_i$.
The highest weight of an irreducible, finite-dimensional representation of $\spin(n-1)$ can be expressed in terms of the fundamental weights (\cite{Branson:92})
\begin{equation*}
\l=(\l_1,\cdots,\l_l),\text{ where }\l\in\mathbb{Z}^l\cup (\tfrac12+\mathbb{Z})^l, \l_1\ge\cdots\ge\l_l\ge 0,\text{ and }l=\tfrac{n-2}{2} .
\end{equation*} 
The spinor bundle, for example, is the bundle associated with the representation with the highest weight $\l=(\tfrac12,\cdots\tfrac12)$.
Let $\G(\l)$ be the space of $\spin(n)$-finite sections of the spinor bundle. Then, by Frobenius reciprocity (\cite{Branson:92}), $\G(\l)$ is the direct sum of the irreducible $\spin(n)$ representations
\begin{equation*}
\G(\l)=\bigoplus_{j=0}^\infty V(\tfrac12+j,\tfrac12,\cdots,\tfrac12,\tfrac{\e}{2})=:\bigoplus_{j=0}^\infty V(j,\e),\; \e=\pm 1.
\end{equation*}
The Dirac operator $D$ acts as $\e\cdot J$ on $V(j,\e)$, where $J=\tfrac{n-1}{2}+j$ (\cite{Branson:92}).
\\
On $S^1$, we are interested in the nonstandard spin structure. The $\spin(2)$-finite sections $e^{\sqrt{-1}ft}$ of the Dirac operator $\pt$ are $4\pi$-periodic ($f\in \tfrac12+\mathbb{Z}$). We will use the notation $(f)$ for this type of $\spin(2)$-finite space. So, the Dirac operator acts as $\sqrt{-1}f$ on $(f)$ over $S^1$. 
\\
Now, on $M=S^1\times S^{n-1}$, a spinor can be viewed as a pair of $t$ ($S^1$ parameter) dependent spinors on $\s1$ (\cite{BH:06}). We define
\begin{equation*}
\a_0=\begin{pmatrix}0&1\\1&0\end{pmatrix},\quad \a_i=\begin{pmatrix}\g_i&0\\0&-\g_i\end{pmatrix}\text{ for }i=1,\dots,n-1.
\end{equation*} 
Using the Clifford relation (\ref{Cliff}) on $\s1$ , it is easy to check that, on the spinor bundle $\S_M$ on $M$,
\begin{equation*}
\a_a\a_b+\a_b\a_a=-2g_{ab}\id_{\S_M},
\end{equation*}
where $g$ is the metric tensor on $M$.
\\
We can choose a complex number $u$ so that the chirality operator on $\s1$, $\c_{_S}=u\g_1\cdots\g_{n-1}$ becomes the identity. That is, $\c_{_S}=\id_\S$ on $\s1$. We define the chirality operator on $M$ as $\c_{_M}=\eye\a_0\c_{_S}$ so that $\c_{_M}^2=\id_{\S_M}$. Since $n$ is even, we have the chirality splitting $\S_M=\S_M^{+1}\os\S_M^{-1}$ of the spinors on $M$ into $\pm 1$ eigenspinors of $\c_{_M}$.
\begin{equation*}
\c_{_M}\spfy=\pm\spfy\implies \spfy=\spm\text{ for }\X=\pm 1.
\end{equation*}
Then the Dirac operator on $M$ is
\begin{equation}\label{dirac}
\dc=\a^a\nabla_a=\ptl_t\begin{pmatrix}0&1\\1&0\end{pmatrix}+D\begin{pmatrix}1&0\\0&-1\end{pmatrix}.
\end{equation}
Note that $\pt$ and $D=\g^i\N_i$ are Dirac operators on $S^1$ and $\s1$, respectively. The Dirac operator $\dc$ changes chirality. That is, $\dc$ maps $\X$-spinors to $-\X$-spinors.
\begin{theorem}
The square of the Dirac operator acts as a constant
\begin{align}
\dc^2\Phi&=(\ptl^2_t+D^2)\id_{\S_M^\X}\Phi=(-f^2+J^2)\Phi\\
& \text{ on }V(f,j,\X,\e)=(f)\otimes V(j,\e)\text{ type }K-\text{finite }\X-\text{spinors},\label{dirac2}
\end{align} 
where $J=\tfrac{n-1}{2}+j$ and $V(f,j,\X,\e)$ is the $K=\spin(2)\times\spin(n)$-finite space consisting of $\Phi=e^{\sqrt{-1}ft}\spm$ for $\f\in V(j,\e)$.
\end{theorem}
\section{Intertwining relations}
Let $(\x_{-1},\x_0,\x_1,\dots\x_n)$ be the homogeneous coordinates of $M$ with $\x_{-1}=\sin t$ and $\x_0=\cos t$. And let $\x_1=\cos\r$ and complete $\r$ to a spherical angular coordinates on $\s1$. 
Let $w_0:=\cos t$, $w_1:=\cos\r$, $Y_0:=\sin t\partial t$ and $Y_1:=\sin\r\partial\r$. Then, $Y_0$ and $Y_1$ are conformal vector fields with conformal factors $w_0$ and $w_1$ on $S^1$ and $S^{n-1}$, respectively (\cite{Branson:87}). That is, with respect to the Lie derivative $\mathcal{L}$ on $S^1$ and $\s1$, 
\begin{equation*}
\mathcal{L}_{_{Y_0}} g_{_{S^1}}=2w_0g_{_{S^1}}\text{ and } \mathcal{L}_{_{Y_1}} g_{_{\s1}}=2w_1g_{_{\s1}}.
\end{equation*}
We define
\begin{equation*}
T:=w_1Y_0+w_0Y_1 \text{ and }\vp:=w_0w_1.
\end{equation*}
Then $\Lt g=2\vp g$, so $T$ is a conformal vector field on $M$ with conformal factor $\vp$.
\\
Let $A$ be an intertwinor of order $2r$ on the spinor bundle satisfying the intertwining relation (\cite{Branson:87}).
\begin{equation}\label{int-rel}
A\left(\Lt+\left(\frac{n}{2}-r\right)\vp\right)=\left(\Lt+\left(\frac{n}{2}+r\right)\vp\right)A\, .
\end{equation}
The Dirac operator $\dc$, for example, is an intertwinor of order $1$.
Note that both $A$ and $\a^0$ exchange the chirality of spinors. So we consider two invariant operators and their intertwining relations for $\Xi=\pm 1$.
\begin{equation*}
\a^0A:\S_M^\Xi\to\S_M^\Xi\text{ and }A\a^0:\S_M^\Xi\to\S_M^\Xi.
\end{equation*}
\begin{lemma}\label{comp-rel}The intertwining relations for $\a^0A$ and $A\a^0$ are
\begin{align}
\a^0A([N,\vp]+P-2r\vp)=([N,\vp]-P+2r\vp)\a^0A,\label{1st}\\
A\a^0([N,\vp]-P-2r\vp)=([N,\vp]+P+2r\vp)A\a^0,\label{2nd}
\end{align}
where
$P:=\sin t\sin\r\a^0\a^1$, $N:=-\pt^2+\N^*\N_{\s1}$, Riemannain $\N^*\N$ on M and $[,]$ is the commutator.
\end{lemma}
\begin{proof} We note that
\begin{equation}\label{l3.1.1}
\begin{split}
\Lt \a^0&=(\Lt\a)(dt)+\a(\Lt dt)=-\vp\a^0+\vp\a^0-\sin t\sin\r\a^1\\
&=-\sin t\sin\r\a^1=\sin t\sin\r\a^0\a^1\a^0=P\a^0,
\end{split}
\end{equation}
\begin{equation}\label{l3.1.2}
\begin{split}
\Lt-\nabla_T&=-\tfrac{1}{8}\left(\nabla_a T_b-\N_b T_a\right)\a^a\a^b\text{ by Kosmann \cite{Kosmann:72}}\\
&=-\tfrac{1}{8}(dT_\flat)_{ab}\a^a\a^b=-\tfrac{1}{2}\sin t\sin\r \a^1\a^0=\tfrac12 P,\\
\end{split}
\end{equation}
\begin{equation}\label{l3.1.3}
\begin{split}
\nabla_T+\tfrac{n}{2}\vp&=w_1Y_0+\tfrac12\vp+w_0\left(\N_{Y_1}+\tfrac{n-1}2w_1\right)\\
&=\tfrac12[-\pt^2,\vp]+\tfrac12{w_0}[\N^*\N_{\s1},w_1]=\tfrac12[N,\vp],\\
&\text{ since }\N_{Y_1}+\tfrac{n-1}2w_1=\tfrac12[\N^*\N_{\s1},w_1] \text{ (\cite{Hong:23})}
\end{split}
\end{equation}
and by (\ref{int-rel}) and (\ref{l3.1.1}),
\begin{equation*}
\begin{split}
\left(\Lt+\left(\tfrac{n}{2}+r\right)\vp\right)\a^0A&=\a^0\left(\Lt+\left(\tfrac{n}{2}+r\right)\vp\right)A+P\a^0A\\
&=\a^0A\left(\Lt+\left(\tfrac{n}{2}-r\right)\vp\right)+P\a^0A .
\end{split}
\end{equation*}
The first equation (\ref{1st}) now follows by (\ref{l3.1.2}) and (\ref{l3.1.3}) and the second equation (\ref{2nd}) follows similarly.
\end{proof}
\noindent Let $\f\in V(j,\e)$. Then the action of the conformal factor $w_1$ can be written as a sum of the adjacent $\spin(n)$-types determined by the tensor product decomposition (selection rule) (\cite{BO:06}) 
\begin{equation}\label{select}
V(1,0,\ldots 0)\otimes V(j,\e)\cong_{\spin(n)}=V(j+1,\e)\oplus V(j-1,\e)\oplus V(f,-\e).
\end{equation}
Thus we get
\begin{equation*}
w_1\f=\f_{j+1}+\f_{j-1}+\f_{-\e}\in V(j+1,\e)\oplus V(j-1,\e)\oplus V(j,-\e).
\end{equation*}
Similarly, for $e^{\sqrt{-1}ft}\in (f)$,
\begin{equation*}
w_0e^{\sqrt{-1}ft}=\tfrac12 e^{\sqrt{-1}(f+1)t}+\tfrac12 e^{\sqrt{-1}(f-1)t}\in (f+1)\oplus (f-1).
\end{equation*}
So, in our case, for $\Phi\in V(f,j,\X,\e)=:V_O$, we have
\begin{equation}
\vp\Phi=\Phi_{NW}+\Phi_{NE}+\Phi_{W}+\Phi_{E}+\Phi_{SW}+\Phi_{SE},\\
\end{equation}
where
$$
\begin{array}{ll}
\Phi_{NW}\in V(f-1,j+1,\X,\e)=:V_{NW},&\Phi_{NE}\in V(f+1,j+1,\X,\e)=:V_{NE},\\
\Phi_{W}\in V(f-1,j,\X,-\e)=:V_{W},&\Phi_{E}\in V(f+1,j,\X,-\e)=:V_{E},\\
\Phi_{SW}\in V(f-1,j-1,\X,\e)=:V_{SW},&\Phi_{SE}\in V(f+1,j-1,\X,\e)=:V_{SE}.  
\end{array}
$$
We apply the relations in the Lemma 3.1 to $\Phi\in V_O$ and project them into each of the six summands in the above.
Let $V'$ be one of the possible six target spaces. Then
\begin{equation}\label{proj1}
\begin{split}
&\proj_{V'}[N,\vp]\Phi=N\proj_{V'}\vp\Phi-\proj_{V'}\vp N\Phi\\
&\quad =(N(V')-N(V_o))\proj_{V'}\vp\Phi=(f'^2-f^2+J'^2-J^2)\proj_{V'}\vp\Phi,
\end{split}
\end{equation}
where $N(V')$ and $N(V_O)$ are $N$ evaluated at $V'$ and $V_O$, respectively and $J'^2$ and $J^2$ are $D^2$ evaluated at $V'$ and $V_O$, respectively. Here we use the Lichnerowicz formula (\cite{Lich:63}) $D^2=\N^*\N_{\s1}+\tfrac{(n-1)(n-2)}{4}$ to compute $[\N^*\N_{\s1},w_1]$ part.
\\
We also compute
\begin{equation}\label{proj2}
\begin{split}
&\proj_{V'}P\Phi=\proj_{V'}\sin t\sin\r\a^0\a^1\Phi=-\X\sqrt{-1}\proj_{V'}[\pt,[D,\vp]]\Phi\\
&\quad=\X(f'-f)(\e'J'-\e J)\proj_{V'}\vp\Phi,
\end{split}
\end{equation}
where $\e' J'$ and $\e J$ are $D$ evaluated at $V'$ and $V_O$, respectively.

\begin{lemma}\label{transition} Let $\m(V)$ and $\nu(V)$ be the eigenvalues of $\a^0A$ and $A\a^0$ on $V$ type $K=\spin(2)\times\spin(n)$ finite space, respectively. Then, with respect to the diagrams
\begin{equation*}
\setlength{\extrarowheight}{7pt}
\left(\begin{array}{cc}
\frac{\m(V_{NW})}{\m(V_O)}&\frac{\m(V_{NE})}{\m(V_O)}\\
\frac{\m(V_{W})}{\m(V_O)}&\frac{\m(V_{E})}{\m(V_O)}\\
\frac{\m(V_{SW})}{\m(V_O)}&\frac{\m(V_{SE})}{\m(V_O)}\\
\end{array}\right)\text{ and }
\left(\begin{array}{cc}
\frac{\nu(V_{NW})}{\nu(V_O)}&\frac{\nu(V_{NE})}{\nu(V_O)}\\
\frac{\nu(V_{W})}{\nu(V_O)}&\frac{\nu(V_{E})}{\nu(V_O)}\\
\frac{\nu(V_{SW})}{\nu(V_O)}&\frac{\nu(V_{SE})}{\nu(V_O)}\\
\end{array}\right),
\end{equation*}
we get
\begin{equation*}
\setlength{\extrarowheight}{7pt}
\left(\begin{array}{cc}
\frac{-f+J+1+r+\Xi\e/2}{-f+J+1-r-\Xi\e/2}&\frac{f+J+1+r-\Xi\e/2}{f+J+1-r+\Xi\e/2}\\
\frac{-f+1/2+r-\Xi\e J}{-f+1/2-r+\Xi\e J}&\frac{f+1/2+r+\Xi\e J}{f+1/2-r-\Xi\e J}\\
\frac{-f-J+1+r-\Xi\e/2}{-f-J+1-r+\Xi\e/2}&\frac{f-J+1+r+\Xi\e/2}{f-J+1-r-\Xi\e/2}
\end{array}\right)\text{ and }
\left(\begin{array}{cc}
\frac{-f+J+1+r-\Xi\e/2}{-f+J+1-r+\Xi\e/2}&\frac{f+J+1+r+\Xi\e/2}{f+J+1-r-\Xi\e/2}\\
\frac{-f+1/2+r+\Xi\e J}{-f+1/2-r-\Xi\e J}&\frac{f+1/2+r-\Xi\e J}{f+1/2-r+\Xi\e J}\\
\frac{-f-J+1+r+\Xi\e/2}{-f-J+1-r-\Xi\e/2}&\frac{f-J+1+r-\Xi\e/2}{f-J+1-r+\Xi\e/2}
\end{array}\right).
\end{equation*}
\end{lemma}
\begin{proof}
It is a consequence of straightforward computations using Lemma \ref{comp-rel}, (\ref{proj1}) and (\ref{proj2}).
\end{proof}
\begin{theorem} On $V(f,j,\X,\e)$ the opeartor $\a^0A$ acts as a constant multiple of
\begin{equation}\label{al}
Z_L(r,f,j,\Xi,\e)=\Xi\e\cdot\textstyle\frac{\G\left(\frac12(f+J+1+r-\Xi\e/2)\right)\G\left(\frac12(-f+J+1+r+\Xi\e/2)\right)}{\G\left(\frac12(f+J+1-r+\Xi\e/2)\right)\G\left(\frac12(-f+J+1-r-\Xi\e/2)\right)}
\end{equation}
and $A\a^0$ acts as a constant multiple of 
\begin{equation}\label{Rop}
\begin{split}
Z_R(r,f,j,\Xi,\e)&=\Xi\e\cdot\textstyle\frac{\G\left(\frac12(f+J+1+r+\Xi\e/2)\right)\G\left(\frac12(-f+J+1+r-\Xi\e/2)\right)}{\G\left(\frac12(f+J+1-r-\Xi\e/2)\right)\G\left(\frac12(-f+J+1-r+\Xi\e/2)\right)}\\
&=-Z_L(r,f,j,\X,-\e).
\end{split}
\end{equation}
\end{theorem}
\begin{proof}
We use Lemma \ref{transition} and the symmetry
\begin{equation*}
\m\text{ at }V(\tfrac12,0,\X,\e)=-\m\text{ at }V(-\tfrac12,0,\X,-\e)
\end{equation*}
together with a normalization at $V(\tfrac12,0,1,\pm 1)$. It is clear from the symmetry that $Z_R(r,f,j,\X,\e)=-Z_L(r,f,j,\X,-\e)$.
\end{proof}
\begin{remark}
The operator $A$ changes chirality but $A^2=A\a^0\cdot\a^0A$ acts as a constant multiple of 
\begin{equation*}
Z_L(r,f,j,\Xi,\e)\cdot Z_R(r,f,j,\Xi,\e) \text{ on }V(f,j,\X,\e).
\end{equation*}
For instant, if $r=\frac12$, on $V(f,j,\X,\e)$,
\begin{align*}
&Z_L(\tfrac12,f,j,\Xi,\e)\cdot Z_R(\tfrac12,f,j,\Xi,\e)=\tfrac14(-f^2+J^2) \text{ and }\\
&\dc^2 \text{ acts as }-f^2+J^2 \text{ by (\ref{dirac2})}.
\end{align*}
\end{remark}
\begin{remark}
The $K=\spin(2)\times\spin(n)$ decomposition of $\X$-spinors has two $\gg=\gso(2,n)=\mathfrak{k}\oplus\mathfrak{s}$ invariant subspaces (\cite{BO:06}) due to the selection rule (\ref{select})
\begin{equation*}
V^{\text{odd}}=\bigoplus_{\substack{\scriptstyle f+j+\e/2\\\text{ odd}}}V(f,j,\X,\e)\quad\text{and}\quad V^{\text{even}}=\bigoplus_{\substack{\scriptstyle f+j+\e/2\\\text{ even}}}V(f,j,\X,\e),
\end{equation*}
where $\mathfrak{k}=\gso(2)\times\gso(n)$ and $\mathfrak{s}$ is the Cartan complement of $\mathfrak{k}$ so that
\begin{equation*}
\bigoplus_{\substack{\scriptstyle f\in 1/2+\mathbb{Z}\\j\in\mathbb{N}}}V(f,j,\X,\e)=V^{\text{odd}}\oplus V^{\text{even}}.
\end{equation*}
\end{remark}
\noindent 
Now we consider all odd order operators $\a^0A$ with $r=k+\tfrac12$, $k\in\mathbb{N}$. The operator $\a^0A$ acts as a constant multiple of (\ref{al}) and in this case,
\begin{align}\label{zl}
Z_L(k+\tfrac12,f,j,\Xi,\e)&=\Xi\e\cdot\left(\tfrac{1}{2}\right)^{2k+1}\cdot\prod_{l=0}^{k-\frac12-\frac{\Xi\e}{2}}(f+J-(k-\tfrac12-\tfrac{\Xi\e}{2})+2l)\nonumber\\&\cdot\prod_{m=0}^{k-\frac12+\frac{\Xi\e}{2}}(-f+J-(k-\tfrac12+\tfrac{\Xi\e}{2})+2m) .
\end{align}
Note that the second product in the above equation (\ref{zl}) can be written by symmetry  as 
\begin{equation}\label{2ndprod}
\begin{split}
&\prod_{m=0}^{k-\frac12+\frac{\Xi\e}{2}}(-f+J-(k-\tfrac12+\tfrac{\Xi\e}{2})+2m)\\
&\quad\quad =(-1)^{k+\tfrac12+\tfrac{\X\e}2}\prod_{m=0}^{k-\frac12+\frac{\Xi\e}{2}}(f-J-(k-\tfrac12+\tfrac{\Xi\e}{2})+2m).
\end{split}
\end{equation}
On the other hand, from (\ref{dirac}), we have
\begin{equation*}
\a^0\dc=\pt\cdot I+D\cdot\left(\begin{array}{cc}0&-1\\1&0\end{array}\right)\text{ and }\dc\a^0=\pt\cdot I-D\cdot\left(\begin{array}{cc}0&-1\\1&0\end{array}\right)
\end{equation*}
implying that, on $V(f,j,\X,\e)$,
\begin{equation}\label{LRD}
-\sqrt{-1}\a^0\dc\text{ acts as }f-\X\e J\text{ and }-\sqrt{-1}\dc\a^0\text{ acts as }f+\X\e J.
\end{equation}
Define
\begin{equation*}
Z(2k+1)=\prod_{l=0}^{k}(-\sqrt{-1}\a^0\dc-k+2l)\cdot\prod_{m=0}^{k-1}(-\sqrt{-1}\dc\a^0-(k-1)+2m) .
\end{equation*}
Then, by (\ref{2ndprod}) and (\ref{LRD}),
\begin{equation*}
Z(2k+1)=(-1)^{k+1}2^{2k+1}\cdot Z_L(k+\tfrac12,f,j,\X,\e).
\end{equation*}
The operator $Z(2k+1)$ satisfies the first equation of the Lemma \ref{comp-rel} of order $2r=2k+1$.
\begin{theorem} 
The order $2k+1$ conformally covariant differential operator $A_{2k+1}$ on the spinor bundle over $M$,
\begin{equation*}
\begin{split}
&A_{2k+1}=\sqrt{-1}\a^0\cdot Z(2k+1)\\
&\quad=\sqrt{-1}\a^0\cdot\prod_{l=0}^{k}(-\sqrt{-1}\a^0\dc-k+2l)\cdot\prod_{m=0}^{k-1}(-\sqrt{-1}\dc\a^0-(k-1)+2m),
\end{split}
\end{equation*}
has the leading term $(-1)^k\cdot\dc^{2k+1}$. And the square of $A_{2k+1}$ acts as
\begin{equation*}
\prod_{l=-k}^k (f+J+l)(-f+J+l)\text{ on }V(f,j,\X,\e).
\end{equation*}
\end{theorem}
\begin{proof}
Since $Z(2k+1)$ satisfies (\ref{1st}), it is clear that $A(2k+1)$ satisfies the intertwining relation (\ref{int-rel}) with $2r=2k+1$. Thus it is a conformally covariant differential operator of order $2k+1$. For the leading term of $A_{2k+1}$, we note that 
\begin{equation*}
\begin{split}
&(-1)^k\a^0(\a^0\dc)^{k+1}(\dc\a^0)^k=(-1)^k\dc(\a^0\dc)^k(\dc\a^0)^k\\
&\quad =(-1)^k\cdot\dc^{2k+1}, \text{ since }\a^0\dc\text{ and }\dc\a^0 \text{ commute}.
\end{split}
\end{equation*}
Finally, 
\begin{equation*}
\begin{split}
&A_{2k+1}^2=\a^0A_{2k+1}A_{2k+1}\a^0=-2^{4k+2}Z_L(k+\tfrac12,f,j,\X.\e)Z_R(k+\tfrac12.f,j,\X,\e)\\
&\quad =2^{4k+2}Z_L(k+\tfrac12,f,j,\X.\e)Z_L(k+\tfrac12,f,j,\X,-\e)\text{ by (\ref{Rop})}\\
&=\prod_{l=0}^{k-\frac12-\frac{\Xi\e}{2}}(f+J-(k-\tfrac12-\tfrac{\Xi\e}{2})+2l)\prod_{m=0}^{k-\frac12+\frac{\Xi\e}{2}}(-f+J-(k-\tfrac12+\tfrac{\Xi\e}{2})+2m)\\
&\cdot\prod_{l=0}^{k-\frac12+\frac{\Xi\e}{2}}(f+J-(k-\tfrac12+\tfrac{\Xi\e}{2})+2l)\prod_{m=0}^{k-\frac12-\frac{\Xi\e}{2}}(-f+J-(k-\tfrac12-\tfrac{\Xi\e}{2})+2m)\\
&=\prod_{l=-k}^k (f+J+l)(-f+J+l).
\end{split}
\end{equation*}
\end{proof}
\noindent The null space $\cN(A_{2k+1})=\cN(Z(2k+1))$ of $A_{2k+1}$ is an invariant subspace for the $(\gg,K)$ representation $(U_{-k-1/2},u_{-k-1/2})$, where $u_{-k-1/2}$ acts on the spinor $\F$ as 
\begin{equation*}
u_{-k-1/2}(h)\F=\W^{-k-\tfrac12+\tfrac n2}_h h\cdot\F
\end{equation*}
for all conformal transformation $h:M\to M$ with $h\cdot g=\W_h^2 g$. Here $h\cdot$ is the natural action of $h$ on tensor-spinors. So, $h\cdot g=(h^{-1})^*g$ and $h\cdot\F=\F\circ h$. And $U_{-k-1/2}$ acts as
\begin{equation*}
U_{-k-1/2}(X)\F=\left(\cL_X+(-k-\tfrac12+\tfrac n2)\vp_X\right)\F
\end{equation*}
for all conformal vector field $X$ with conformal factor $W_X$. The operator $A_{2k+1}$ intertwines between the representations $(U_{-k-1/2},u_{-k-1.2})$ and $(U_{k+1/2},u_{k+1.2})$ (\cite{Branson:87}).
\\
The selection rule (\ref{select}) implies that $\cN(A_{2k+1})\subset V^{\text{odd}}$ if $(-1)^{k-n/2}=-1$ and $\cN(A_{2k+1})\subset V^{\text{even}}$ if $(-1)^{k-n/2}=1$.
\\
We define the spaces of generalized positive and negative frequency waves,
\begin{align*}
&W^+_{2k+1}=\mkern-40mu\bigoplus_{f=J-\left(k-\tfrac12+\tfrac{\X\e}2\right)+2m}\mkern-40mu V(f,j,\X,\e)\text{ for }m=0,\ldots, k-\tfrac12+\tfrac{\X\e}2\text{ and }\\
&W^-_{2k+1}=\mkern-40mu\bigoplus_{-f=J-\left(k-\tfrac12-\tfrac{\X\e}{2}\right)+2l}\mkern-40mu V(f,j,\X,\e)\text{ for }l=0,\ldots, k-\tfrac12-\tfrac{\X\e}2.
\end{align*}
The selection rule also shows that these are $(\gg,K)$ invariant. The intersection $F:=W^+_{2k+1}\cap W^-_{2k+1}$ is also invariant and nonempty if $k\ge\tfrac n2$. If $k=\tfrac n2$, we have $F=V(-\tfrac{\X\e}2,0,\X,\e)$.


\vspace{1cm}
\noindent Doojin Hong\\
Department of Mathematics\\
University of North Dakota\\
Grand Forks, ND 58202, USA\\
Email: doojin.hong@und.edu
\end{document}